\input ./style/arxiv-vmsta.cfg
\documentclass[numbers,compress,v1.0.1]{vmsta}
\usepackage{vtexbibtags}
\usepackage{authorquery}

\volume{5}
\issue{4}
\pubyear{2018}
\firstpage{457}
\lastpage{470}
\aid{VMSTA112}
\doi{10.15559/18-VMSTA112}
\articletype{research-article}

\startlocaldefs

\allowdisplaybreaks

\newtheorem{theorem}{Theorem}[section]
\newtheorem{proposition}{Proposition}[section]

\newtheorem{remark}{Remark}[section]

\newtheorem{definition}{Definition}[section]

\newcommand{\de}{\mathrm{d}}

\newcommand{\bb}{\mathbb}

\newcommand{\rr}{\mathbb{R}}
\newcommand{\cor}{\mathcal}

\endlocaldefs

\begin{aqf}
\querytext{Q1}{Please check the comma before $\de r$.}
\querytext{Q2}{Note that in what follows $(Y^n(t),t\geq0)$ is meant in the singular context.}
\end{aqf}
\begin{document}

\begin{frontmatter}
\pretitle{Research Article}

\title{Stochastic models associated to a Nonlocal Porous Medium Equation}

\author{\inits{A.}\fnms{Alessandro}~\snm{De Gregorio}\ead[label=e1]{alessandro.degregorio@uniroma1.it}\orcid{0000-0003-0809-6046}}
\address{Dipartimento di Scienze Statistiche,
\institution{``Sapienza'' University of Rome},
P.le Aldo Moro, 5 - 00185, Rome, \cny{Italy}}



\markboth{A. De Gregorio}{Stochastic models associated to a Nonlocal
Porous Medium Equation}

\begin{abstract}
The nonlocal porous medium equation considered in this paper is a
degenerate nonlinear evolution equation involving a space
pseudo-differential operator of fractional order. This space-fractional
equation admits an explicit, nonnegative, compactly supported weak
solution representing a probability density function. In this paper we
analyze the link between isotropic transport processes, or random
flights, and the nonlocal porous medium equation. In particular, we
focus our attention on the interpretation of the weak solution of the
nonlinear diffusion equation by means of random flights.
\end{abstract}
\begin{keywords}
\kwd{Anomalous diffusions}
\kwd{finite speed of propagation}
\kwd{fractional gradient}
\kwd{random flights}
\end{keywords}
%

\received{\sday{24} \smonth{4} \syear{2018}}
\revised{\sday{19} \smonth{7} \syear{2018}}
\accepted{\sday{24} \smonth{8} \syear{2018}}
\publishedonline{\sday{19} \smonth{9} \syear{2018}}
\end{frontmatter}

\section{Introduction}

We deal with a Nonlocal Porous Medium Equation (NPME) studied in \cite
{biler0,biler}, given by the following degenerate nonlinear and
nonlocal evolution equation
\begin{equation}
\label{eq:nonloceq} \partial_t u=\text{div} \bigl(|u|\nabla^{\alpha-1}
\bigl(|u|^{m-2}u\bigr) \bigr),\quad m>1,\,\alpha\in(0,2],\,t>0,
\end{equation}
subject to the initial condition
\begin{equation}
\label{initcon} u(x,0)=u_0(x),
\end{equation}
where $u:=u(x,t)$, with $x:=(x_1,\ldots,x_d)\in\bb R^d,d\geq1$, is a
scalar function defined on $\rr^d\times\rr^+$ and $\partial_t:=\partial
/\partial t$. The pseudo-differential operator $\nabla^{\alpha-1}$ is
the fractional gradient denoting the nonlocal operator defined as
$\nabla^{\alpha-1}u:=\cor F^{-1}(i\xi||\xi||^{\alpha-2}\cor Fu)$,
where the Fourier transform $\cor F$ and the inverse transform $\cor
F^{-1}$ of a function $v\in L^1(\bb R^d)$ are defined by
\[
\cor F v(\xi)=\frac{1}{(2\pi)^{d/2}}\int_{\bb R^d} v(\xi)
e^{i x\cdot
\xi}\de x,\quad\cor F^{-1} v(\xi)=\frac{1}{(2\pi)^{d/2}}\int
_{\bb R^d} v(\xi) e^{-i x\cdot\xi}\de x,
\]
with $\xi\in\bb R^d$. This notation highlights that $\nabla^{\alpha-1}$
is a pseudo-differential (vector-valued) operator of order $\alpha-1$.
Equivalently, we can define $\nabla^{\alpha-1}$ as $\nabla(-\Delta
)^{\frac{\alpha}2-1}$, where $(-\Delta)^{\frac{\alpha}2}u=\cor F^{-1}(||\xi
||^\alpha\cor Fu)$ is the fractional Laplace operator; i.e. a Fourier
multiplier 
with the symbol $||\xi||^\alpha$. We observe that
$\nabla^1=\nabla$ is the classical gradient and that div$(\nabla^{\alpha
-1})=\nabla^{\frac{\alpha}2}\cdot\nabla^{\frac{\alpha}2}=-(-\Delta)^{\frac{\alpha}2}$.
Another equivalent definition of the fractional gradient
$\nabla^{\alpha-1}$ involves the Riesz potential; that is $\nabla
^{\alpha-1}=\nabla I_{2-\alpha}$ where $I_\beta=(-\Delta)^{-\frac{\beta}
2}$ is a Fourier multiplier with symbol $||\xi||^{-\beta},\beta\in
(0,2)$ (for more details on this point see \cite{biler0,biler}).

In \cite{biler0,biler}, explicit and compactly supported nonnegative
self-similar solutions of \eqref{eq:nonloceq} are constructed. These
explicit solutions generalize the well-known
Barenblatt--Kompanets--Zel'dovich--Pattle solutions of the porous medium
equation \eqref{eq:mpme} below. Furthermore, the authors proved the
existence of sign-changing weak solution to the Cauchy problem \eqref
{eq:nonloceq}--\eqref{initcon} for $u_0(x)\in L^1(\rr^d)$, and the
hypercontractivity $L^1\mapsto L^p$ estimates.

By exploiting Darcy's law, it is possible to interpret the equation
\eqref{eq:nonloceq} as a transport equation $\partial_t u=\text
{div} (|u|{\bf v} )$, where ${\bf v}:=\nabla{\bf p}:=\nabla
I_{2-\alpha}(|u|^{m-2}u)$ is a vector velocity field with nonlocal and
nonlinear pressure ${\bf p}$ in the case of nonnegative initial data.
We observe that the fractional operator $\nabla I_{2-\alpha}$
represents the long-range diffusion effects. The one-dimensional version
of the pseudo-differential equation \eqref{eq:nonloceq} describes the
dynamics of dislocations in crystals (see \cite{bilerkarch}).

For $\alpha=2$, \eqref{eq:nonloceq} becomes the classical nonlinear
porous medium equation
\begin{equation}
\label{eq:pme} \partial_t u=\text{div} \bigl(|u|\nabla
\bigl(|u|^{m-2}u\bigr) \bigr)=\text {div} \bigl((m-1)|u|^{m-1}
\nabla u \bigr).
\end{equation}
If we restrict our attention to nonnegative solution $u(x,t)$, the
equation \eqref{eq:pme} becomes
\begin{equation}
\label{eq:mpme} \partial_t u=\frac{m-1}{m}\Delta
\bigl(u^m\bigr),
\end{equation}
which is usually adopted to model the flow of a gas through a porous medium.
The reader interested in the theory of porous medium equation can
consult, for instance, \cite{vazquez}.

Other types of nonlocal porous medium equations have been proposed in
literature. For instance, \cite{caff,caff2} introduced the porous
medium equation with fractional diffusion effects
\begin{equation}
\label{eq: caff-vaz} \partial_t u=\text{div}(u\nabla p),
\end{equation}
with nonlocal pressure $p:=(-\Delta)^{-s}u, 0<s<1$, and $u\geq0$. For
$\alpha=2-2s\in(0,2)$ we obtain the equation \eqref{eq:nonloceq} with
$m=2$; i.e. $\partial_t u=\text{div}(u\nabla^{\alpha-1}u)$. In \cite
{stan} the nonlinear diffusion equation \eqref{eq: caff-vaz} is
generalized as follows
\begin{equation}
\label{eq:stan} \partial_t u=\text{div}\bigl(u^{m-1}\nabla p
\bigr),\quad u(x,t)\geq0, x\in \mathbb R^d,t>0,
\end{equation}
with $m>1$ and initial condition $u(x,0)=u_0(x)$ which is nonnegative
bounded with compact support or fast decaying at infinity. The main
contribution in \cite{stan} concerns the study of the property of
finite/infinite speed of propagation of the solutions to \eqref
{eq:stan} with varying $m$.

The following equation
\begin{equation}
\label{eq: fracvazq}\partial_t u=-(-\Delta)^{\frac{\alpha}
2}
\bigl(|u|^{m-1}u\bigr),\quad\alpha\in(0,2),
\end{equation}
is studied in \cite{vazquez2}, where it is also proved that the
self-similar solutions of \eqref{eq: fracvazq} enjoy the
$L^1$-contraction property and then they are unique. Nevertheless,
these solutions are not compactly supported. Explicit self-similar
solutions to \eqref{eq:stan} and \eqref{eq: fracvazq} have been
obtained by \cite{huang} for particular values of $m$.

The main goal of this paper is to investigate the relationship between
\eqref{eq:nonloceq} and some random models. In particular, we focus our
attention to the probabilistic interpretations of the weak solution to
NPME. The idea to study stochastic processes associated to the
classical porous medium equation \eqref{eq:mpme} was developed by
different authors; see, for instance, \cite{inoue,inoue2,inoue3,ekhaus,feng,jou,phil}.
In the
listed papers the authors
introduced different types of Markov chains on lattice and interacting
particle systems having a dynamic which macroscopically converges to
the solution of \eqref{eq:mpme}. By \cite{getoor}, the Barenblatt
solution of \eqref{eq:mpme} can be viewed as the mean of the first
passage time of a symmetric stable process to exterior of a ball. In
\cite{ben}, the authors provided a probabilistic interpretation of
\eqref{eq:mpme} in terms of stochastic differential equations.
Recently, \cite{dgo2} highlighted the connection between \eqref
{eq:mpme} and the Euler--Poisson--Darboux equations by taking into
account time-rescaled random flights.

Up to our knowledge, this paper is the first attempt concerning the
probabilistic interpretation of the fractional porous medium equation
\eqref{eq:nonloceq}. Similarly to \cite{dgo2}, we can exploit
stochastic models defined by continuous-time random walks in $\rr
^d,d\geq1$, arising in the description of the displacements of a
particle choosing uniformly its directions; i.e. the so-called
isotropic transport processes or random flights. In a suitable
time-rescaled frame, the probability law of the above processes is
given by the solution \eqref{eq:weaksol} below. Therefore, this paper
represents a generalizations of some results contained in \cite{dgo2}.
We point out that the proposed random processes recover some features
of the Barenblatt weak solution \eqref{eq:weaksol} to nonlinear
evolution equations like finite speed of propagation and the anomalous
diffusivity. For this reason the random flights seem to represent a
natural way to describe the real phenomena studied by means of \eqref
{eq:nonloceq}.

In Section \ref{sec:weak}, we recall the definition of weak solution
to \eqref{eq:nonloceq} as well as its basic properties. In Section \ref
{sec:rf} the isotropic transport processes are introduced. Furthermore,
Section \ref{sec:rf} contains our main results; i.e. Propositions \ref
{teo:sppme}, \ref{teo:radial}. From these propositions we are
able to give a reasonable interpretation of the solutions to \eqref
{eq:nonloceq}. In the last section we sum up the main contribution of
the paper.

\section{A review on the weak solutions to NPME}\label{sec:weak}
Let us recall the definition of weak solution to the nonlocal operator
equation \eqref{eq:nonloceq} and its main properties (see \cite{biler}).
\begin{definition}\label{def}
A function $u: \bb R^d\times(0,T)\to\bb R$ is a weak solution to the
Cauchy problem \eqref{eq:nonloceq}--\eqref{initcon} in $\bb R^d\times
(0,T)$ if $u\in L^1(\bb R^d\times(0,T)), \nabla^{\alpha
-1}(|u|^{m-2}u)\in L_{loc}^1(\bb R^d\times(0,T))$, $|u|\nabla^{\alpha
-1}(|u|^{m-2}u)\in L_{loc}^1(\bb R^d\times(0,T))$ and
\[
\int_{\bb R^d}\int_0^T \bigl(u
\partial_t \varphi-|u|\nabla^{\alpha
-1}\bigl(|u|^{m-2}u
\bigr)\cdot\nabla\varphi \bigr)\de t\de x+\int_{\bb
R^d}u_0(x)
\varphi(x,0)\de x=0,
\]
where $\varphi\in C^\infty(\bb R^d\times(0,T))\cap C^1(\overline{\bb
R^d\times(0,T)})$
has a compact support in the
space variable $x$ and vanishes near $t=T$.
\end{definition}

Let $(x)_+:=\max(x,0)$. The following theorem, proved in \cite{biler},
represents our starting point.
\begin{theorem}
Let $\alpha\in(0,2]$ and $m>1$. A weak solution in the sense of
Definition \ref{def} in $(\eta,T)\times\bb R^d$, for every $0<\eta
<T<\infty$, is given by the function $u:\bb R^d\times(0,T)\to[0,\infty
)$ defined as
\begin{align}
\label{eq:weaksol} u(x,t)= Ct^{-d\beta} \biggl(1-k^{\frac{2}\alpha}
\frac{||x||^2}{t^{2\beta}} \biggr)_+^{\frac{\alpha}{2(m-1)}},
\end{align}
where $\beta:=\beta(\alpha,d,m):=\frac{1}{d(m-1)+\alpha}$,
\begin{align*}
k&:=k(\alpha, d):=\frac{d\varGamma(d/2)}{(d(m-1)+\alpha)2^\alpha\varGamma
(1+\frac{\alpha}2)\varGamma(\frac{d+\alpha}{2})},\\
C&:=C(\alpha,d,m):=\frac{\varGamma(\frac{d}2+\frac{\alpha}{2(m-1)}+1)k^{\frac
{ d}{\alpha}}}{\pi^{\frac{d}2}\varGamma(\frac{\alpha}{2(m-1)}+1)}.
\end{align*}
Furthermore, $u(x,t)$ is the pointwise solution of the equation \eqref
{eq:nonloceq} for $||x||\neq\mathrm{c} t^\beta$ and is $\min\{\frac
{\alpha}{m-1},1\}$-H\"older continuos at $||x||=\mathrm{c} t^\beta$,
where $\mathrm{c}:=\mathrm c(\alpha,d):=1/k^{\frac{1}\alpha}$.
\end{theorem}

It is worth to mention that the family of functions \eqref{eq:weaksol}
represents a class of nonnegative compactly supported solutions of
\eqref{eq:nonloceq}. Moreover, \eqref{eq:weaksol} is a self-similar
solution under a suitable space-time rescaling; i.e.
\[
u(x,t)=L^{d\beta}u\bigl(L^\beta x,Lt\bigr), \quad L>0.
\]
It is crucial to observe that the constant $C$ appearing in \eqref
{eq:weaksol} guarantees the mass conservation
\[
\int_{\bb R^d}u(x,t)\de x=\int_{\bb R^d}u_0(x)
\de x=1,
\]
or equivalently
\[
\frac{\de}{\de t}\int_{\bb R^d}u(x,t)\de x=\frac{\de}{\de t}
\int_{\bb
R^d}u_0(x)\de x=0,
\]
and then $u(x,t)$ (as well as $u_0(x)$) is a probability density
function with compact support $\overline{\bb B}_{\mathrm{c} t^\beta}:=\{
x\in\rr^d: ||x||\leq\mathrm{c} t^\beta\}$. By setting $R^2=\frac
{1}{k^{2/\alpha}}$, the solution \eqref{eq:weaksol} coincides with
(2.4) in \cite{biler}.\goodbreak

We point out that NPME has the property of finite speed of propagation.
We are able to explain this property as follows. The solution to NPME
is a continuous function $u(x,t)$ such that for any $t > 0$ the profile
$u( \cdot, t )$ is nonnegative, bounded and compactly supported. Hence,
the support expands eventually to penetrate the whole space, but it is
bounded at any fixed time. Therefore, for fixed $t>0$, the support of
\eqref{eq:weaksol} is given by the closed ball $\overline{\bb B}_{\mathrm{c} t^\beta}$,
while the free boundary (that is the set separating the region where
the solution is positive) is given by the sphere $\mathbb S^{d-1}_{\mathrm{c} t^\beta
}:=\{x\in\mathbb R^d:||x||= \mathrm{c} t^\beta\}$. 

The finite speed of propagation of NPME is in contrast with the
infinite speed of propagation of the classical heat equation; that is, a
nonnegative solution of the heat equation is positive everywhere in
$\mathbb R^d$.

\begin{remark}\label{solpme}
For $\alpha=2$, the solution \eqref{eq:weaksol} becomes the
Barenblatt--Kompanets--\break Zel'dovich--Pattle solution to the porous medium
equation \eqref{eq:mpme} supplemented with the initial condition
$u(x,0)=\delta(x)$ (see, for instance, \cite{vazquez}).
\end{remark}

\begin{remark}\label{inv}
From Theorem \ref{teo:sppme} it follows that $(u(x,t),t\geq0)$ is a class
of rotationally invariant functions; that is, let $O(d)$ be the group of
$d\times d$ orthogonal matrices acting in $\mathbb R^d$, then we have that
$u(M^Tx,t)=u(x,t)=u(||x||,t)$,
where $M\in O(d)$.
\end{remark}

The next proposition contains the explicit Fourier transform of \eqref
{eq:weaksol}. A similar result has been already proved, for instance,
in \cite{biler}, Lemma 4.1.

\begin{proposition}\label{propcf}
The Fourier transform of the probability density function $u(x,t)$
given by \eqref{eq:weaksol} is equal to
\begin{align}
\label{eq:cfbsol} \hat u(\xi,t)&:= \cor Fu(\xi,t)\notag
\\
&= \frac{1}{(2\pi)^{\frac d2}}\biggl(\frac{2k^{1/\alpha}}{t^\beta||\xi||} \biggr)^{\frac{d}{2}+\frac
{\alpha}{2(m-1)}}\varGamma \biggl(
\frac{d}{2}+\frac{\alpha}{2(m-1)}+1 \biggr)J_{\frac{d}{2}+\frac{\alpha}{2(m-1)}} \biggl(
\frac{||\xi||t^{\beta
}}{k^{1/\alpha}} \biggr),
\end{align}
where $\xi\in\mathbb R^d,d\geq1,$ and $J_\mu(x)=\sum_{k=0}^\infty(-1)^k\frac
{(x/2)^{2k+\mu}}{k!\varGamma(k+\mu+1)}$, with $\mu\in\mathbb R$, is the
Bessel function.
\end{proposition}
\begin{proof}
We prove the theorem for $d\geq 2.$ The case $d=1$ follows by simple calculations.
Let $\sigma$ be the measure on $\mathbb S_{1}^{d-1}$. We recall
that (see (2.12), p.~690, \cite{dgo}),
\begin{equation}
\label{eq:int} \int_{\mathbb S_1^{d-1}}e^{i\rho\xi\cdot\theta} \de\sigma({\bf
\theta})=(2\pi)^{d/2}\frac{J_{\frac{d}2-1}(\rho||\xi||)}{(\rho||\xi
||)^{\frac{d}2-1}}
\end{equation}
One has that
\begin{align*}
\hat u(\xi,t)&=\frac{1}{(2\pi)^{\frac d2}}\int_{\mathbb R^d}e^{i\xi\cdot x}u(x,t)\de x
\\
&=(\text{by Remark \ref{inv}})
\\
&=\frac{1}{(2\pi)^{\frac d2}}\int_0^{\frac{t^{\beta}}{k^{1/\alpha}}}\rho^{d-1}
Ct^{-d\beta} \biggl(1-\frac{k^{2/\alpha}\rho^2}{t^{2\beta}} \biggr)^{\frac{\alpha
}{2(m-1)}}\de\rho\int
_{\mathbb S_1^{d-1}}e^{i\rho\xi\cdot\theta} \de\sigma({\bf\theta})
\\
&=(\text{by}\, \eqref{eq:int})
\\
&=\int_0^{\frac{t^{\beta}}{k^{1/\alpha}}}
\rho^{d-1} Ct^{-d\beta} \biggl(1-\frac{k^{2/\alpha}\rho^2}{t^{2\beta}}
\biggr)^{\frac
{\alpha}{2(m-1)}}\frac{J_{\frac{d}2-1}(\rho||\xi||)}{(\rho||\xi
||)^{\frac{d}2-1}}\de\rho
\\
&=\frac{ Ct^{\beta(1-\frac{d}2)}}{(k^{1/\alpha})^{\frac{d}2+1}||\xi||^{\frac{d}2-1}}\int_0^1
\bigl(1-w^2\bigr)^{\frac{\alpha}{2(m-1)}} w^{d/2}J_{d/2-1}
\biggl(\frac{||\xi||t^{\beta}}{k^{1/\alpha}}w \biggr)\de w.
\end{align*}
In view of formula 6.567(1) on p.~688 of \cite{gr},
\begin{equation}
\label{eq:gr} \int_0^1x^{\nu+1}
\bigl(1-x^2\bigr)^\mu J_\nu(bx)\de
x=2^\mu\varGamma(\mu+1)b^{-(\mu
+1)}J_{\nu+\mu+1}(b)
\end{equation}
where $b>0$, Re$\nu>-1$, Re$\mu>-1$, we obtain \eqref{eq:cfbsol}.
\end{proof}


\section{Isotropic transport processes related to NPME}\label{sec:rf}
In this section, we analyze the link between the weak solution of the
nonlocal equation \eqref{eq:nonloceq} and the transport processes. We
follow the 
approach developed in \cite{dgo2}. Let us start with
introducing isotropic transport processes and recalling their main features.


An isotropic transport process, also called random flight, is a
continuous-time random walk in $\mathbb R^d$ described by a particle
starting at the origin with a randomly chosen direction and with finite
speed $c>0$. The direction of the particle changes whenever a collision
with some scattered obstacles in the environment happens and then a new
direction of motion is taken. For $d\geq2$, all the
directions are
independent and identically distributed. The directions are chosen
uniformly on the sphere $\mathbb S_1^{d-1}=\{x\in\mathbb R^d:||x||=1\}$.
For $d=1$, we have two possible directions alternatively taken by the
moving particle. The random flights have been studied, for instance, in
\cite{stadje,stadje2,dgos,lecaer,lecaer2,dgo,DG12,ghosh,lp}.
Recently, in \cite{garra,garra2} the relationship between the isotropic
transport processes and some fractional Klein--Gordon equations has been
analyzed. Furthermore, stochastic models like 
random flights are
associated to the Euler--Poisson--Darboux partial differential equations
as argued in \cite{garra3}.

Rigorously speaking, we introduce the isotropic transport processes as follows.
Let $(T_k,k\in\bb N_0)$ be a sequence of random arrival epochs with
$T_0:=0$. Furthermore, let $(V_k,k\in\bb N_0)$ be a sequence of random
variables defined for $d=1$, by $V_k:=V(0)(-1)^k$, where V(0) is a
uniform r.v. on $\{-1,+1\}$, while for $d\geq2$ they are independent
$(\bb S_1^{d-1},\mathcal B(\bb S_1^{d-1}))$-valued random variables
where $\mathcal B(\bb S_1^{d-1})$ denotes the Borel class on $\bb
S_1^{d-1} $.
We assume that during the interval $[0,t]$ the particle
takes a new direction, $V_0,V_1,\ldots,V_n$, $n+1$ times at random moments
$T_0,T_1,\ldots,T_n$, respectively. Therefore, we can define an isotropic
random flight on $(\varOmega, (\cor F_t^n,t\geq0))$ as follows
\[
X^n:=\bigl(X^n(t)=\bigl(X_1^n(t),
\ldots,X_d^n(t)\bigr), t\geq0\bigr),\quad
X^n(0)=0,\quad n\in\mathbb N,
\]
where $X^n(t)$ stands for the position, at time $t\geq0$, reached by
the moving particle according to the mechanism described above and
$(V^n(t),t\geq0)$ is the jump process
\[
V^n(t):=V_k, \quad T_k\leq
t<T_{k+1},
\]
with $0\leq k\leq n$; i.e.
\begin{equation}
\label{eq:definitionaddfun} X^n(t):=c\int_0^t
V^n(s)\de s=c\sum_{k=0}^{ n-1}V_k(T_{k+1}-T_{k})+c
(t-T_{ n})V_{ n},\quad n\in\mathbb N.
\end{equation}
$X^n$ is adapted to the filtration $ (\cor F_t^n,t\geq0)$ where
\[
\cor F_t^n:=\sigma\bigl((T_k\leq t)
\cap(V_0,V_1,\ldots,V_k)\in B, \forall B\in
\mathfrak B^{\otimes k+1},0\leq k\leq n\bigr),
\]
where $\mathfrak B:=\{-1,+1\}$, if $d=1$, or $\mathfrak B:=\mathcal
B(\bb S_1^{d-1})$, if $d>1$. Therefore $X^n(t)$ represents a random
motion with \xch{finite}{finte} velocity $c$ and $X^n(t)\in\bb B_{ct}$ a.s. for a
fixed $t>0$.
The components of $X^n(t)$ can be written explicitly as in formula
(1.6) of \cite{dgo}.
Important assumptions in our paper are:
the random vector of the renewal times $(\tau_1,\ldots,\tau_n)$, where $\tau
_{k+1}:=T_{k+1}-T_k$, has the joint density equal to
\begin{equation}
\label{eq:jointdis1} f_1(\tau_1,\ldots,\tau_n)=
\frac{n!}{t^n}\,1_{S_n}(\tau_1,\ldots,
\tau_n),\quad \text{for}\ d=1,
\end{equation}
\noindent
and
\begin{equation}
\label{eq:jointdis2} f_2(\tau_1,\ldots,\tau_n)=
\frac{\varGamma((n+1)(d-1))}{(\varGamma
(d-1))^{n+1}}\frac{1}{t^{(n+1)(d-1)-1}} \Biggl(\prod_{j=1}^{n+1}
\tau _j^{d-2} \Biggr)\,1_{S_n}(\tau_1,
\ldots,\tau_n),
\end{equation}
for $d\geq2$,
or
\begin{equation}
\label{eq:jointdis2bis} f_3(\tau_1,\ldots,\tau_n)=
\frac{\varGamma((n+1)(\frac{d}2-1))}{(\varGamma(\frac{d}2-1))^{n+1}}\frac{1}{t^{(n+1)(\frac{d}2-1)-1}} \Biggl(\prod_{j=1}^{n+1}
\tau_j^{\frac{d}2-2} \Biggr)\,1_{S_n}(\tau_1,
\ldots,\tau_n),
\end{equation}
for $d\geq3$,
where
\begin{align*}
S_n:= \Biggl\{(&\tau_1,\ldots,\tau_n)\in\bb
R^d:0<\tau_j<t-\sum_{k=0}^{j-1}
\tau _k,\\
& 1\leq j\leq n, \tau_0=0, \tau_{n+1}=t-
\sum_{j=1}^n\tau_j \Biggr\}.
\end{align*}
The distributions \eqref{eq:jointdis2} and \eqref{eq:jointdis2bis} are
rescaled Dirichlet distributions, with parameters $(d-1,\ldots,d-1),\,
d\geq2$, and $(\frac{d}2-1,\ldots,\frac{d}2-1),\, d\geq3$, respectively.
Generalized versions of the Dirichlet density functions \eqref
{eq:jointdis1} and \eqref{eq:jointdis2} have been used in \cite{DG14}
to generalize the family of random walks defined above.

In the one-dimensional case the process \eqref{eq:definitionaddfun} is
the well-known telegraph process and admits the density given by (see \cite{dgos})
\begin{equation}
\label{eq:disttp} \frac{P(X_1^n(t)\in\de x_1)}{\de x_1}= %
\begin{cases}
\frac{\varGamma(n+1)}{(\varGamma(\frac{n+1}{2}))^22^{n}ct} (1-\frac
{x_1^2}{c^2t^2} )_+^{\frac{n-1}{2}},& n\, \text{odd},\\
\frac{\varGamma(n+1)}{\varGamma(\frac{n}2+1)\varGamma(\frac{n}2)2^{n}ct}
(1-\frac{x_1^2}{c^2t^2} )_+^{\frac{n}2-1},& n\, \text{even}.
\end{cases} %
\end{equation}
We observe that for $n$ odd, we have that
\[
P\bigl(X_1^n(t)\in\de x_1\bigr)=P
\bigl(X_1^{n+1}(t)\in\de x_1\bigr).
\]
Under the assumptions \eqref{eq:jointdis2} and \eqref{eq:jointdis2bis},
\cite{dgo} 
provides (Theorem 2 in \cite{dgo}) the explicit density
functions of the random flights $X^n(t)$; that is,
\begin{equation}
\label{eq:distrf} \frac{P(X^n(t)\in\de x)}{\de x}= %
\begin{cases}
\frac{\varGamma(\frac{n+1}{2}(d-1)+\frac{1}2)}{\varGamma(\frac{n}{2}(d-1))\pi
^{\frac{d}2}(ct)^d} (1-\frac{||x||^2}{c^2t^2} )_+^{\frac{n}2(d-1)-1},&\text{if \eqref{eq:jointdis2} holds} , \\
\frac{\varGamma((n+1)(\frac{d}2-1)+1)}{\varGamma(n(\frac{d}2-1))\pi^{\frac{d}2}(ct)^d} (1-\frac{||x||^2}{c^2t^2} )_+^{n(\frac{d}2-1)-1},&\text{if \eqref{eq:jointdis2bis} holds}.
\end{cases} %
\end{equation}

\begin{remark}
It is easy to check that the sequence of random flights $X^n,n\in
\mathbb N$, admits the following scaling property
\[
P\bigl(aX^n( t/a)\in\de x\bigr)=P\bigl(X^n(t)\in\de x
\bigr),\quad a>0.
\]
\end{remark}

Hereafter, we discuss the main results of the paper; i.e. Propositions
\ref{teo:sppme}, \ref{teo:radial} below. Therefore, we provide a
reasonable probabilistic interpretation of the weak solution \eqref
{eq:weaksol} in terms of a time-rescaled random flights. From the
features of $X^n$ it emerges that the random flights share with \eqref
{eq:weaksol} the crucial property of finite speed of propagation in the
space. For this reason the transport process \eqref
{eq:definitionaddfun} seems to represent a fine choice to model
phenomena described by nonlinear diffusion equation with nonlocal
pressure \eqref{eq:nonloceq}. Our first result is the following theorem
and it represents a generalization of Theorem 1 in \cite{dgo2}.

\begin{proposition}\label{teo:sppme}
Let $Y^n:=(Y^n(t),t\geq0), n\in\bb N$, be the sequence of random
flights $Y^n(t):= X^n(t^\beta)$, with speed $\mathrm{c}:=\mathrm
c(\alpha,d):=1/k^{\frac{1}\alpha}$. $Y^n$ is adapted to the filtration
$(\cor G_t^n,t\geq0)$, where $\cor G_t^n:=\cor F_{t^\beta}^n$,
progressively measurable and
\[
P\bigl(Y^n(t)\in\de x\bigr)=u(x,t) \de x, \quad t>0,
\]
where $u(x,t)$ is \xch{the}{the the} weak solution \eqref{eq:weaksol} to the
equation \eqref{eq:nonloceq}.
The relationships between the number $n$ of changes of velocity of
$Y^n$ and the parameters $m>1$ and $\alpha\in(0,2]$ of NPME, are given by:
\begin{itemize}
\item[(i)] for $d=1$,
\[
m= %
\begin{cases}
\frac{\alpha}{n-1}+1=\frac{\alpha}{2k}+1,& n=2k+1,\\
\frac{\alpha}{n-2}+1=\frac{\alpha}{2k}+1,& n=2k+2,
\end{cases} %
\quad k\geq1;
\]
\item[(ii)] for $d\geq2$, \eqref{eq:jointdis2} holds and $ m=\frac
{\alpha}{n(d-1)-2}+1$ with $d>\frac{2}n+1$;
\item[(iii)] for $d\geq3$, \eqref{eq:jointdis2bis} holds and $m=\frac
{\alpha}{n(d-2)-2}+1$ with $d>\frac{2}n+2$.
\end{itemize}
\end{proposition}
\begin{proof}

Let $n\in\mathbb N$.  We observe that path map
\[
t\mapsto Y^n(t,\omega)=\mathrm c\int_0^{t^\beta}
V^n(s,\omega)\de s,\quad\omega\in\varOmega,
\]
is continuous and then $Y^n$ is a continuous process. Therefore, $Y^n$
is progressively measurable if it is adapted to $(\cor G_t^n,t\geq0)$
(see, e.g., Proposition 1.13, \cite{karat}). Let $t^\beta>0$, $(s,\omega
)\mapsto V^n(s,\omega),\omega\in\varOmega,s\leq t^\beta$ is a $\mathcal
B([0,t^{\beta}])\otimes\mathcal G_t^n$-measurable function. Hence, by
Fubini's theorem one has that the map $\omega\mapsto\mathrm c \int_0^{t^\beta} V^n(s,\omega)\de s$ is $\mathcal G_t^n$-measurable and
then the process $Y^n$ is adapted to the filtration $(\cor G_t^n,t\geq0)$.

By rescaling the time coordinate as follows
\[
t':=t^\beta,
\]
the solution \eqref{eq:weaksol} to NPME becomes
\begin{equation}
\label{eq:fundsol2} u\bigl(x,t'\bigr)=\frac{\varGamma(\frac{d}2+\frac{\alpha}{2(m-1)}+1)}{\varGamma(\frac
{\alpha}{2(m-1)}+1)\pi^{\frac{d}2}}
\frac{1}{(\mathrm{c}t')^{d}} \biggl(1-\frac{||x||^2}{(\mathrm{c}t')^2} \biggr)_+^{\frac{\alpha}{2(m-1)}},
\end{equation}
where $\mathrm{c}:=\mathrm c(\alpha,d):=1/k^{\frac{1}\alpha}$.

Let us deal with a telegraph process defined by \eqref
{eq:definitionaddfun} with time scale $t'$ and speed $\mathrm{c}$. By
exploiting the duplication formula for the Gamma function we can write
the solution \eqref{eq:fundsol2} for $d=1$ as follows
\begin{equation}
\label{eq:epd1} u\bigl(x_1,t'\bigr)=\frac{\varGamma(2+\frac{\alpha}{m-1})2^{1-2(\frac{\alpha
}{2(m-1)}+1)}}{(\varGamma(\frac{\alpha}{2(m-1)}+1))^2}
\frac{1}{\mathrm
{c}t'} \biggl(1-\frac{x_1^2}{(\mathrm{c}t')^2} \biggr)_+^{\frac{\alpha}{2(m-1)}}.
\end{equation}
For
\begin{equation*}
\frac{\alpha}{2(m-1)}=\frac{n-1}{2}, \quad\text{that is}\quad m=
\frac{\alpha}{n-1}+1,
\end{equation*}
the solution \eqref{eq:epd1} coincides with the first part of \eqref
{eq:disttp}, while for
\begin{equation*}
\frac{\alpha}{2(m-1)}=\frac{n}{2}-1, \quad\text{that is}\quad m=
\frac{\alpha}{n-2}+1,
\end{equation*}
the solution \eqref{eq:epd1} coincides with the second part of \eqref{eq:disttp}.
For $n>2$, in both cases $m\in(1,\infty)$. Therefore, we can conclude that
\begin{equation*}
P\bigl(X_1^n\bigl(t'\bigr)\in\de
x_1\bigr)= P\bigl(X_1^{n+1}\bigl(t'
\bigr)\in\de x_1\bigr)=u\bigl(x_1,t'\bigr)
\de x_1.
\end{equation*}

Now,
let us consider a random flight defined in $\mathbb R^d, d\geq2$, by
\eqref{eq:definitionaddfun} with time scale $t'$ and speed $\mathrm c$
defined above. Under the assumption \eqref{eq:jointdis2}, for
\begin{equation*}
\frac{\alpha}{2(m-1)}=\frac{n}2(d-1)-1, \quad\text{that is}\quad m=
\frac
{\alpha}{n(d-1)-2}+1,
\end{equation*}
the function \eqref{eq:fundsol2} coincides with the first part of \eqref
{eq:distrf}. Since $m\in(1,\infty)$, we infer that
\begin{equation}
\label{eq:ineq} d>\frac{2}n+1.
\end{equation}
For $d=2$ the inequality \eqref{eq:ineq} holds for $n\geq3$; for
$d=3$, it holds for $n\geq2$; for $d>3$, \eqref{eq:ineq} holds for all
$n\geq1$.
Therefore, under the condition \eqref{eq:ineq}
\begin{equation*}
P\bigl(X^n\bigl(t'\bigr)\in\de x\bigr)= u
\bigl(x,t'\bigr) \de x.
\end{equation*}
Analogously, under the assumption \eqref{eq:jointdis2bis}, for
\begin{equation*}
\frac{\alpha}{2(m-1)}= n \biggl(\frac{d}2-1 \biggr)-1, \quad\text{that is}
\quad m=\frac{\alpha}{n(d-2)-2}+1,
\end{equation*}
the function \eqref{eq:fundsol2} coincides with the second part of \eqref
{eq:distrf}. Since $m\in(1,\infty)$, we infer that
\begin{equation}
\label{eq:ineq2} d>\frac{2}n+2.
\end{equation}
For $d=3$ the inequality \eqref{eq:ineq} holds for $n\geq3$; for
$d=4$, it holds for $n\geq2$; for $d>4$, \eqref{eq:ineq} holds for all
$n\geq1$.
Therefore, under the condition \eqref{eq:ineq2}
\begin{equation*}
P\bigl(X^n\bigl(t'\bigr)\in\de x\bigr)= u
\bigl(x,t'\bigr) \de x.\qedhere
\end{equation*}
\end{proof}

To enhance the features of the random models $Y^n, n\geq 1$, it is
useful to introduce the Euclidean distance process $\mathsf
{R}^n:=(\mathsf{R}^n(t),t\geq0)$; that is $\mathsf
R^n(t):=||Y^n(t)||$. For a fixed $t\geq0$, $\mathsf{R}^n(t)\in
[0,\mathrm c t^\beta]$ a.s. The next result will be useful for arguing
on the anomalous diffusivity of $Y^n$.
\begin{proposition}\label{teo:radial}
Under the conditions (i), (ii) and (iii) of Proposition \ref
{teo:sppme}, the following results hold:

1) the probability density function of\/ $\mathsf{R}^n$ becomes:
\begin{align}
\label{eq:distpdf} \frac{P(\mathsf{R}^n(t)\in\de r)}{\de r} &=\frac{2\varGamma(\frac{d}2+\frac{\alpha}{2(m-1)}+1)}{\varGamma(\frac{\alpha
}{2(m-1)}+1)\varGamma(\frac{d}2)}\frac{r^{d-1}}{ (\mathrm c t^\beta
 )^{d}}
\biggl(1-\frac{r^2}{\mathrm c^2t^{2\beta}} \biggr)_+^{\frac
{\alpha}{2(m-1)}},\quad t> 0;
\end{align}

2) let $p\geq1$ and $d\geq2$; then
\begin{align}
\label{eq:mom} E\bigl(\mathsf{R}^n(t)\bigr)^p=
\frac{\varGamma(\frac{d}2+\frac{\alpha
}{2(m-1)}+1)\varGamma(\frac{d+p}{2})}{\varGamma(\frac{\alpha}{2(m-1)}+1+\frac
{d+p}{2})\varGamma(\frac{d}2)} \bigl(\mathrm ct^\beta \bigr)^{p},
\end{align}
while for $d=1$
\begin{align}
\label{eq:mom1} E\bigl(\mathsf{R}^n(t)\bigr)= %
\begin{cases}
0,& p\,\text{odd},\\
\frac{\varGamma(\frac{1}2+\frac{\alpha}{2(m-1)}+1)\varGamma(\frac
{1+p}{2})}{\varGamma(\frac{\alpha}{2(m-1)}+1+\frac{p+1}{2})\sqrt\pi}
(\mathrm ct^\beta )^{p}, & p\,\text{even};
\end{cases}
\end{align}

3) the rescaled process $ (\frac{X^n(t^\beta)}{\mathrm ct^\beta
},t\geq0 )$ has the distribution
law independent from the time $t$ and
with compact support $\overline{\bb B}_1$; i.e.
\begin{equation*}
\label{eq:rescproc} w(x,t):=\frac{P (\frac{X^n(t^\beta)}{\mathrm ct^\beta}\in\de
x )}{\de x}=\frac{\varGamma(\frac{d}2+\frac{\alpha}{2(m-1)}+1)}{\pi
^{d/2}\varGamma(\frac{\alpha}{2(m-1)}+1)}\bigl(1-||x||^2
\bigr)_+^{\frac{\alpha}{2(m-1)}}.
\end{equation*}
Furthermore
\begin{align*}
\label{eq:rescproccf} \hat w(\xi,t)&:= \cor Fw(\xi,t)\\*
&=\frac{1}{(2\pi)^{\frac{d}{2}}} \biggl(\frac{2}{||\xi||}
\biggr)^{\frac
{d}{2}+\frac{\alpha}{2(m-1)}}\varGamma \biggl(\frac{d}{2}+\frac{\alpha
}{2(m-1)}+1
\biggr)J_{\frac{d}{2}+\frac{\alpha}{2(m-1)}} (||\xi || );
\end{align*}

4) the rescaled distance process $ (\frac{\mathsf{R}^n(t)}{\mathrm c
t^\beta},t\geq0 )$ admits probability density function given by a
Beta r.v. with parameters $\frac{d}2$ and $\frac{\alpha}{2(m-1)}+1$.
\end{proposition}
\begin{proof}
1) By exploiting Remark \ref{inv} and $P(\mathsf{R}^n(t) \leq
r)=P(Y^n(t)\in\mathbb B_r)$, it is not hard to prove that
\begin{align*}
P\bigl(\mathsf{R}^n(t)\in\de r\bigr)&=\text{area} \bigl(\bb
S_1^{d-1}\bigr)r^{d-1}u(r,t) \de r,
\end{align*}
\querymark{Q1}where $\text{area} (\bb S_1^{d-1})=2\pi^{d/2}/\varGamma(d/2)$, and then
\eqref{eq:distpdf} immediately follows.

2) From point 1), we have
\begin{align*}
E\bigl(\mathsf{R}^n(t)\bigr)^p&=\int
_0^\infty\text{area}\bigl(\bb S_1^{d-1}
\bigr)r^{p+d-1}u(r,t) \de r\notag
\\
&=\frac{2\varGamma(\frac{d}2+\frac{\alpha}{2(m-1)}+1)}{\varGamma(\frac{\alpha
}{2(m-1)}+1)\varGamma(\frac{d}2)}\int_0^{\mathrm ct^\beta}
\frac
{r^{p+d-1}}{ (\mathrm ct^\beta )^{d}} \biggl(1-\frac
{r^2}{\mathrm c^2t^{2\beta}} \biggr)^{\frac{\alpha}{2(m-1)}} \de r\notag
\\
&=\frac{\varGamma(\frac{d}2+\frac{\alpha}{2(m-1)}+1)}{\varGamma(\frac{\alpha
}{2(m-1)}+1)\varGamma(\frac{d}2)} \bigl(\mathrm ct^\beta \bigr)^{p}\int
_0^{1}w^{\frac{d+p}{2}-1} (1-w )^{\frac{\alpha}{2(m-1)}} \de
w\notag
\\
&=\frac{\varGamma(\frac{d}2+\frac{\alpha}{2(m-1)}+1)\varGamma(\frac
{d+p}{2})}{\varGamma(\frac{\alpha}{2(m-1)}+1+\frac{d+p}{2})\varGamma(\frac{d}2)} \bigl(\mathrm ct^\beta \bigr)^{p}.
\end{align*}
For $d=1$ the result \eqref{eq:mom1} follows by similar calculations.

3) For fixed $t> 0$, the result \eqref{eq:rescproc} is derived from \eqref
{eq:fundsol2}, by applying the Jacobian theorem to the bijection $g:\rr
^d \to\rr^d$ with $g(x)=\frac{1}{\mathrm c t^\beta}x$. By the same
calculations leading to \eqref{eq:cfbsol}, we can prove that the
Fourier transform $\hat w(\xi,t)$ holds true.

4) It is an immediate consequence of the point 1).
\end{proof}

The Barenblatt--Kompanets--Zel'dovich--Pattle solution to the classical
PME does not spread in the space linearly over the time and then we can
argue that the phenomena described by the equation \eqref{eq:mpme}
represent anomalous diffusion (see, for instance, \cite{vazquez}).
Similar considerations hold for \eqref{eq:weaksol}. By means of Theorem
\ref{teo:radial}, we infer that the stochastic models $Y^n, n\geq 1,$ behave similarly to an anomalous diffusion.
From \eqref{eq:mom} and \eqref{eq:mom1}, we observe that
\[
\text{Var}\bigl(\mathsf{R}^n(t)\bigr)=O\bigl(t^{2\beta}\bigr),
\quad t>0.
\]
For $d=1$, one has $2\beta=\frac{2}{m-1+\alpha}=\frac{4k}{\alpha
(2k+1)}, k\geq1$ (condition (i) in Proposition \ref{teo:sppme}).
Therefore, for a fixed $k\in\mathbb N$, we can find the values $\alpha\in
(0,2]$, such that the process\querymark{Q2} $Y^n$ spreads over the real
line like a sub-diffusion or a super-diffusion. Therefore
$Y^n$ has the following properties:
\begin{itemize}
\item scatters in the space as a sub-diffusion; i.e. $2\beta<1$, if and
only if $\frac{4k}{(2k+1)}<\alpha$;
\item is a super-diffusion process; i.e. $2\beta>1$, if and only if
$\frac{4k}{(2k+1)}>\alpha$;
\item represents a classical diffusion if and only if $\frac{4k}{
(2k+1)}=\alpha$ (i.e. $2\beta=1$).
\end{itemize}
Analogous remarks hold in higher dimensions. For  a fixed $n\in\mathbb N$ and
$d\geq2$ (resp. $d\geq3$), under the condition (ii) (resp. (iii)) in
Proposition \ref{teo:sppme}, the random process $Y^n$ has the following properties:
\begin{itemize}
\item behaves similarly to a sub-diffusion if and only if $\frac
{2n(d-1)-4}{(n+1)(d-1)-1}<\alpha$ (resp. $\frac
{2n(d-2)-4}{(n+1)(d-2)-2}<\alpha$);
\item spreads over the space like a super-diffusion if and only if
$\frac{2n(d-1)-4}{(n+1)(d-1)-1}>\alpha$ (resp. $\frac
{2n(d-2)-4}{(n+1)(d-2)-1}>\alpha$);
\item represents a diffusion if and only if $\frac
{2n(d-1)-4}{(n+1)(d-1)-1}=\alpha$ (resp. $\frac
{2n(d-2)-4}{(n+1)(d-2)-1}=\alpha$).
\end{itemize}

\section{Conclusions}


We are able to provide a probabilistic interpretation of the weak
solution \eqref{eq:weaksol} to NPME. In particular, we deal with random
flight models \eqref{eq:definitionaddfun} with a suitable rescaling of
the time coordinate. These random processes enjoy the main features of
\eqref{eq:weaksol}, at least for particular values of $m$:
\begin{itemize}
\item finite speed of propagation property with compact support given
by a closed ball;
\item spread over the space like $t^{2\beta}$; i.e. anomalous
diffusivity depending on the values of the fractional parameter $\alpha$.
\end{itemize}

In conclusion, the isotropic transport processes seem to describe well
the real phenomena studied by means of the degenerate nonlinear
diffusion equation with fractional pressure \eqref{eq:nonloceq}.

\begin{acknowledgement}[title={Acknowledgments}]
The author wishes to thank both the referees for their comments and
remarks which led to improve the earlier version of the paper.
\end{acknowledgement}




\end{document}